\documentclass[a4paper,12pt]{article}

\usepackage{amsmath}
\usepackage{amssymb}
\usepackage{amsthm}
\usepackage{amsfonts}
\usepackage{mathtools}
\usepackage{ascmac}
\usepackage{spalign}
\usepackage{fancybox}
\usepackage[mathscr]{eucal}
\usepackage{footnpag}
\usepackage{siunitx}
\usepackage{multicol}
\usepackage{bbm}
\usepackage{framed}
\usepackage{here}

\usepackage{graphicx}
\usepackage{xcolor}
\usepackage{wrapfig}
\usepackage{float}

\theoremstyle{definition}
\newtheorem{theorem}{Theorem}[section]
\newtheorem*{theorem*}{Theorem}
\newtheorem{definition}{Definition}[section]
\newtheorem*{definition*}{Definition}
\newtheorem{proposition}{Proposition}[section]
\newtheorem*{proposition*}{Proposition}
\newtheorem{lemma}{Lemma}[section]
\newtheorem*{lemma*}{Lemma}
\newtheorem{corollary}{Corollary}[section]
\newtheorem*{corollary*}{Corollary}
\newtheorem{remark}{Remark}[section]
\newtheorem*{remark*}{Remark}

\newfont{\bg}{cmr9 scaled\magstep4}
\newcommand{\bigzero}{\smash{\lower1.0ex\hbox{\bg 0}}}

\title{Edge-defect matrices and stability of the Kirchhoff index
for complete graphs with deleted edges}

\author{
Shunya Tamura\thanks{
Okegawa City Okegawa West Junior High School,
Saitama, 363-0027, Japan, 
ORCID: 0009-0000-2162-2536; 
e-mail: \texttt{shunya.tamura059@gmail.com}
}
}

\date{}

\begin{document}
\maketitle
%%%%%%%%%%%%%%%%%%%%%%%%%%%%%%%%%%%%%%%%
\begin{abstract}
%%%%%%%%%%%%%%%%%%%%%%%%%%%%%%%%%%%%%%%%

In this paper, we study the effective resistance, the Kirchhoff index,
and the number of spanning trees of the connected graph $K_n-F$,
which is obtained from the complete graph by deleting a set $F$ of $p$ edges.
Let $B$ be the incidence matrix of the deleted edges.
We call the matrix $Q=B^TB$ the edge-defect matrix.
This is a $p\times p$ matrix which records, with signs,
the way in which the deleted edges share their end vertices.

First, we derive a formula for the effective resistance between
any two distinct vertices in terms of the resolvent of the edge-defect matrix.
This reduces the usual computation using the $n\times n$ Laplacian matrix
to a computation using a $p\times p$ matrix corresponding to the number
of deleted edges.
Moreover, by using the eigenvalues of the same matrix,
we give unified formulas for the Kirchhoff index and the number of spanning trees.

Next, we derive a stability identity which exactly describes the excess
from the Xu--Das--Zhang type lower bound.
As a consequence, we show that, in the range where a matching deletion
can be realized, the Kirchhoff index is minimized when the deleted edge set
is a matching.
Furthermore, by using majorization, we prove that, for $p\ge 2$ and
$n\ge \max\{4,2p-1\}$, among all non-matching deleted edge sets,
the minimum is attained only when the deletion graph is isomorphic to
$P_3\cup(p-2)K_2$.
Finally, we apply the obtained formulas to several deletion graphs,
such as matchings, stars, cliques, paths, and cycles.

%%%%%%%%%%%%%%%%%%%%%%%%%%%%%%%%%%%%%%%%
\end{abstract}
%%%%%%%%%%%%%%%%%%%%%%%%%%%%%%%%%%%%%%%%

%%%%%%%%%%%%%%%%%%%%%%%%%%%%%%%%%%%%%%%%


%%%%%%%%%%%%%%%%%%%%%%%%%%%%%%%%%%%%%%%%
\section{Introduction}
\label{sec:introduction}
%%%%%%%%%%%%%%%%%%%%%%%%%%%%%%%%%%%%%%%%

Let $L(G)$ be the Laplacian matrix of a connected graph $G$,
and let $L(G)^+$ be its Moore--Penrose inverse.
When each edge is regarded as a unit resistor, the effective resistance
between two distinct vertices $x,y\in V(G)$ is given by
\[
R_G(x,y)
=
(e_x-e_y)^TL(G)^+(e_x-e_y).
\]
This quantity was introduced by Klein and Randi\'c \cite{KleinRandic1993}
as the resistance distance.
The sum of effective resistances over all unordered pairs of vertices,
\[
Kf(G)
=
\sum_{\{x,y\}\in\binom{V(G)}{2}}R_G(x,y),
\]
is called the Kirchhoff index of $G$.
Let $G$ be a connected graph with $n$ vertices, and let
\[
0=\mu_1<\mu_2\le\cdots\le\mu_n
\]
be the Laplacian eigenvalues of $G$.
Then
\[
Kf(G)
=
n\sum_{i=2}^{n}\frac{1}{\mu_i}
\]
holds \cite{GutmanMohar1996}.
Also, by Kirchhoff's Matrix-Tree Theorem, the number of spanning trees
$\tau(G)$ of $G$ is given by
\[
\tau(G)
=
\frac{1}{n}\prod_{i=2}^{n}\mu_i
\]
\cite{Kirchhoff1847,Bapat2010}.

In this paper, we study the connected graph
\[
G=K_n-F
\]
obtained from the complete graph $K_n$ by deleting a set of edges
\[
F=\{e_1,\ldots,e_p\}\subseteq E(K_n).
\]
Since
\[
R_{K_n}(x,y)=\frac{2}{n}
\]
holds for the complete graph, it is natural to ask how edge deletions
affect the effective resistance, the Kirchhoff index, and the number
of spanning trees.

There are several studies on edge deletions with symmetry.
Chair \cite{Chair2012} expressed the effective resistance of $K_N-C_N$
in terms of Bejaia--Pisa numbers.
Tamura \cite{Tamura2026} and Tamura and Tanaka \cite{TamuraTanaka2026}
studied complete graphs with deleted edges corresponding to cyclic distances,
and obtained effective resistances and the number of spanning trees
by using discrete Fourier analysis.
These methods depend on the cyclic symmetry of the deleted edge set,
and therefore they cannot be directly applied to a general edge set.

We put
\[
V(F)=\bigcup_{e\in F}e,
\qquad
H_F=(V(F),F),
\]
and call $H_F$ the deletion graph of $F$.
We give an arbitrary orientation to each edge of $F$, and let $B$ be
the corresponding incidence matrix.
Furthermore, we put
\[
Q=Q_F=B^TB
\]
and call it the edge-defect matrix.
The matrix $Q$ is the edge Laplacian of the deletion graph $H_F$.
It is a $p\times p$ matrix which records how the deleted edges share
their end vertices.

The purpose of this paper is not merely to introduce this matrix.
We show that the same small matrix $Q$ describes the effective resistance,
the Kirchhoff index, and the number of spanning trees of $K_n-F$
in a unified way.
Moreover, from the eigenvalue distribution of $Q$, we derive stability
and extremal results for the Kirchhoff index.

The first main result of this paper is the following reduction formula.
For two distinct vertices $x,y\in V(G)$, put
\[
\eta_{xy}=B^T(e_x-e_y).
\]
Let
\[
\theta_1,\ldots,\theta_p
\]
be the eigenvalues of $Q$.
Then
\[
R_{K_n-F}(x,y)
=
\frac{2}{n}
+
\frac{1}{n}
\eta_{xy}^T(nI_p-Q)^{-1}\eta_{xy},
\]
\[
Kf(K_n-F)
=
n-1+
\sum_{\alpha=1}^{p}
\frac{\theta_\alpha}{n-\theta_\alpha},
\]
and
\[
\begin{aligned}
\tau(K_n-F)
&=
n^{n-2}
\prod_{\alpha=1}^{p}
\left(
1-\frac{\theta_\alpha}{n}
\right)\\
&=
n^{n-p-2}\det(nI_p-Q)
\end{aligned}
\]
hold.
Thus the three quantities which are usually computed from the
$n\times n$ Laplacian matrix are reduced to computations using
a $p\times p$ matrix corresponding to the number of deleted edges.

The second main result is a complete stability identity for the Kirchhoff
index.
Xu, Das and Zhang \cite{XuDasZhang2016} proved that, when
\[
2\le p\le\left\lfloor\frac{n}{2}\right\rfloor,
\]
we have
\[
Kf(K_n-F)
\ge
n-1+\frac{2p}{n-2},
\]
and equality holds if and only if
\[
H_F\cong pK_2.
\]
In this paper, we refine this lower bound into the identity
\[
\begin{aligned}
Kf(K_n-F)
&=n-1+\frac{2p}{n-2}+\frac{n}{(n-2)^2}
\sum_{\alpha=1}^{p}\frac{(\theta_\alpha-2)^2}{n-\theta_\alpha}.
\end{aligned}
\]
This formula holds for connected $K_n-F$ without imposing an upper bound
on the number of deleted edges, and it exactly describes the excess from
the Xu--Das--Zhang type lower bound.
Moreover, the correction term vanishes if and only if
\[
Q=2I_p,
\]
which is equivalent to
\[
H_F\cong pK_2.
\]

The third main result is the determination of the sharp minimum,
together with the equality condition, among all non-matching deleted
edge sets.
Palacios \cite{Palacios2015} studied the ranking of graphs with diameter
$2$ which have small Kirchhoff indices.
In this paper, we use majorization
\cite{MarshallOlkinArnold2011,Palacios2018}
for the eigenvalue vector of the edge-defect matrix.
We prove that, if $p\ge2$ and
\[
n\ge\max\{4,2p-1\},
\]
then, whenever $H_F$ is not a matching,
\[
Kf(K_n-F)
\ge
n-1
+
\frac{3}{n-3}
+
\frac{1}{n-1}
+
\frac{2(p-2)}{n-2}.
\]
Moreover, equality holds if and only if
\[
H_F\cong P_3\cup(p-2)K_2.
\]
In other words, among all non-matching deleted edge sets, the Kirchhoff
index is minimized by the structure in which exactly one pair of deleted
edges shares a vertex.
This extends the result of Palacios \cite{Palacios2015} for $p=2$,
and the known ranking for $p=3$ in \cite{XuDasZhang2016},
to arbitrary $p$.

Furthermore, when $n\ge2p$, the difference between the minimum attained
by a matching and the minimum attained by a non-matching deleted edge set is
\[
\frac{2n}{(n-1)(n-2)(n-3)}.
\]
This value is independent of the number $p$ of deleted edges, and gives
the first stability gap.

This paper is organized as follows.
In Section \ref{sec:preliminaries}, we prepare the necessary notation,
matrix formulas, and results on majorization.
In Section \ref{sec:resistance-formula}, we prove the resolvent formula
for the effective resistance.
In Section \ref{sec:kirchhoff-spanning}, we derive the reduction formulas
for the Kirchhoff index and the number of spanning trees.
In Section \ref{sec:stability}, we prove the stability identity.
In Section \ref{sec:first-gap}, we determine the minimum among
non-matching deleted edge sets and the first stability gap.
In Section \ref{sec:examples}, we apply our formulas to typical deletion
graphs.
Finally, in Section \ref{sec:conclusion}, we give the conclusion and future
problems.

%%%%%%%%%%%%%%%%%%%%%%%%%%%%%%%%%%%%%%%%
\section{Preliminaries}
\label{sec:preliminaries}
%%%%%%%%%%%%%%%%%%%%%%%%%%%%%%%%%%%%%%%%

Unless otherwise stated, all graphs in this paper are finite, simple,
and undirected.
For a graph $G$, we denote its vertex set and edge set by
$V(G)$ and $E(G)$, respectively.

%%%%%%%%%%%%%%%%%%%%%%%%%%%%%%%%%%%%%%%%
\subsection{Laplacian matrix and effective resistance}
%%%%%%%%%%%%%%%%%%%%%%%%%%%%%%%%%%%%%%%%

Let $G$ be a graph with $n$ vertices.
We denote its adjacency matrix and degree matrix by $A(G)$ and $D(G)$,
respectively.
The Laplacian matrix is defined by
\[
L(G)=D(G)-A(G).
\]
When $G$ is connected, we write its Laplacian eigenvalues as
\[
0=\mu_1<\mu_2\le\cdots\le\mu_n.
\]

Let $L(G)^+$ be the Moore--Penrose inverse of $L(G)$,
and let $e_x$ be the standard basis vector corresponding to a vertex $x$.
For two distinct vertices $x,y\in V(G)$, the effective resistance between
$x$ and $y$ is given by
\[
R_G(x,y)
=
(e_x-e_y)^TL(G)^+(e_x-e_y)
\]
\cite{KleinRandic1993,Bapat2010}.

%%%%%%%%%%%%%%%%%%%%%%%%%%%%%%%%%%%%%%%%
\begin{definition}
\label{def:kirchhoff-index}
The Kirchhoff index of a connected graph $G$ is defined by
\[
Kf(G)
=\sum_{\{x,y\}\in\binom{V(G)}{2}}R_G(x,y).
\]
\end{definition}

We use the following standard spectral formulas
\cite{Kirchhoff1847,GutmanMohar1996,Bapat2010}.

%%%%%%%%%%%%%%%%%%%%%%%%%%%%%%%%%%%%%%%%
\begin{theorem}[Kirchhoff index \cite{GutmanMohar1996} and Matrix--Tree Theorem \cite{Kirchhoff1847}]
\label{thm:kirchhoff-spectral}
\label{thm:matrix-tree}
Let $G$ be a connected graph with $n$ vertices.
Then
\[
Kf(G)
=n\operatorname{tr}L(G)^+
=n\sum_{i=2}^{n}\frac{1}{\mu_i}.
\]
Also, if $\tau(G)$ denotes the number of spanning trees of $G$, then
\[
\tau(G)
=\frac{1}{n}\prod_{i=2}^{n}\mu_i.
\]
\end{theorem}

%%%%%%%%%%%%%%%%%%%%%%%%%%%%%%%%%%%%%%%%
\subsection{Matrix reduction formulas}
%%%%%%%%%%%%%%%%%%%%%%%%%%%%%%%%%%%%%%%%

We use the following standard matrix formulas
\cite{Bhatia1997}.

%%%%%%%%%%%%%%%%%%%%%%%%%%%%%%%%%%%%%%%%
\begin{lemma}
\label{lem:nonzero-spectrum}
Let $B$ be an $n\times p$ real matrix, and put
\[
Q=B^TB.
\]
Then the nonzero eigenvalues of $BB^T$ and $Q$ are the same,
including multiplicities.
\end{lemma}

\begin{proof}
Let $\lambda\ne0$.
First, suppose that
\[
BB^Tx=\lambda x
\]
for some $x \in \mathbb{R}^n\setminus\{\boldsymbol{0}\}$.
Then
\[
B^Tx\ne\boldsymbol{0}.
\]
Indeed, if
\[
B^Tx=\boldsymbol{0},
\]
then
\[
BB^Tx
=
B(B^Tx)
=
\boldsymbol{0},
\]
and hence
\[
\lambda x=\boldsymbol{0}.
\]
Since $\lambda\ne0$, this implies $x=\boldsymbol{0}$, which is a contradiction.

Moreover,
\[
\begin{aligned}
Q(B^Tx)
&=
B^TB(B^Tx)\\
&=
B^T(BB^Tx)\\
&=
B^T(\lambda x)\\
&=
\lambda B^Tx.
\end{aligned}
\]
Thus $B^Tx$ is a $\lambda$-eigenvector of $Q$.
Hence $B^T$ gives a linear map from the $\lambda$-eigenspace of $BB^T$
to the $\lambda$-eigenspace of $Q$.

Conversely, suppose that
\[
Qy=B^TBy=\lambda y
\]
for some $y \in \mathbb{R}^p\setminus\{\boldsymbol{0}\}$.
Then
\[
By\ne\boldsymbol{0}.
\]
Indeed, if
\[
By=\boldsymbol{0},
\]
then
\[
Qy
=
B^TBy
=
\boldsymbol{0},
\]
and hence
\[
\lambda y=\boldsymbol{0}.
\]
Since $\lambda\ne0$, this implies $y=\boldsymbol{0}$, which is a contradiction.

Moreover,
\[
\begin{aligned}
BB^T(By)
&=
B(B^TBy)\\
&=
B(\lambda y)\\
&=
\lambda By.
\end{aligned}
\]
Thus $By$ is a $\lambda$-eigenvector of $BB^T$.

Now, for a $\lambda$-eigenvector $x$ of $BB^T$, we have
\[
B(B^Tx)
=
BB^Tx
=
\lambda x.
\]
Therefore,
\[
\frac{1}{\lambda}B(B^Tx)=x.
\]
Also, for a $\lambda$-eigenvector $y$ of $Q$, we have
\[
B^T(By)
=
B^TBy
=
\lambda y.
\]
Therefore,
\[
B^T\left(
\frac{1}{\lambda}By
\right)
=
y.
\]
Thus $B^T$ and $B/\lambda$ are inverse maps between the two
$\lambda$-eigenspaces.
Hence the dimensions of the $\lambda$-eigenspaces of $BB^T$ and $Q$
are the same.

Since both matrices are real symmetric matrices, the algebraic multiplicity
of each eigenvalue is equal to the dimension of its eigenspace.
Therefore, the nonzero eigenvalues of $BB^T$ and $Q$ are the same,
including multiplicities.
\end{proof}

%%%%%%%%%%%%%%%%%%%%%%%%%%%%%%%%%%%%%%%%
\begin{lemma}
\label{lem:resolvent-reduction}
Let $B$ be an $n\times p$ real matrix, and put
\[
Q=B^TB.
\]
If $nI_p-Q$ is nonsingular, then
\[
(nI_n-BB^T)^{-1}
=
\frac{1}{n}I_n
+
\frac{1}{n}
B(nI_p-Q)^{-1}B^T
\]
holds.
\end{lemma}

\begin{proof}
Put
\[
R=(nI_p-Q)^{-1},
\]
and define
\[
X
=
\frac{1}{n}I_n
+
\frac{1}{n}BRB^T.
\]
We show that $X$ is the inverse of $nI_n-BB^T$.

Using $Q=B^TB$, we have
\[
\begin{aligned}
(nI_n-BB^T)X
&=
(nI_n-BB^T)
\left(
\frac{1}{n}I_n
+
\frac{1}{n}BRB^T
\right)\\
&=
I_n
+
BRB^T
-
\frac{1}{n}BB^T
-
\frac{1}{n}BB^TBRB^T\\
&=
I_n
+
BRB^T
-
\frac{1}{n}BB^T
-
\frac{1}{n}BQRB^T\\
&=
I_n
+
B\left(
R-\frac{1}{n}I_p-\frac{1}{n}QR
\right)B^T.
\end{aligned}
\]

On the other hand, since
\[
R=(nI_p-Q)^{-1},
\]
we have
\[
(nI_p-Q)R=I_p.
\]
Expanding this equality gives
\[
nR-QR=I_p.
\]
Dividing both sides by $n$, we obtain
\[
R-\frac{1}{n}QR
=
\frac{1}{n}I_p.
\]
Therefore,
\[
R-\frac{1}{n}I_p-\frac{1}{n}QR
=
O.
\]
Substituting this into the above expression, we obtain
\[
(nI_n-BB^T)X
=
I_n.
\]

Since both $nI_n-BB^T$ and $X$ are $n\times n$ matrices,
the right inverse $X$ is also a left inverse.
Hence
\[
X=(nI_n-BB^T)^{-1}.
\]
Thus
\[
(nI_n-BB^T)^{-1}
=
\frac{1}{n}I_n
+
\frac{1}{n}
B(nI_p-Q)^{-1}B^T.
\]
\end{proof}

%%%%%%%%%%%%%%%%%%%%%%%%%%%%%%%%%%%%%%%%
\subsection{The edge-defect matrix}
%%%%%%%%%%%%%%%%%%%%%%%%%%%%%%%%%%%%%%%%

Let $F\subseteq E(K_n)$ be
\[
F=\{e_1,\ldots,e_p\},
\]
and put
\[
G=K_n-F.
\]
Also, put
\[
V(F)=\bigcup_{e\in F}e,
\qquad
H_F=(V(F),F),
\]
and call $H_F$ the deletion graph of $F$.

Let
\[
e_\alpha=\{u_\alpha,v_\alpha\}
\]
be a deleted edge.
We give each deleted edge an arbitrary orientation and define
\[
b_\alpha=e_{u_\alpha}-e_{v_\alpha}.
\]
Furthermore, put
\[
B=
\begin{bmatrix}
b_1&\cdots&b_p
\end{bmatrix}.
\]

%%%%%%%%%%%%%%%%%%%%%%%%%%%%%%%%%%%%%%%%
\begin{definition}
\label{def:edge-defect-matrix}
The edge-defect matrix of the deleted edge set $F$ is defined by
\[
Q=Q_F=B^TB.
\]
\end{definition}

The matrix $Q$ is the edge Laplacian obtained from the incidence matrix
of the deletion graph $H_F$.

%%%%%%%%%%%%%%%%%%%%%%%%%%%%%%%%%%%%%%%%
\begin{lemma}
\label{lem:basic-Q}
The edge-defect matrix $Q$ satisfies the following properties.
\begin{enumerate}
\item[(1)]
$Q$ is a real symmetric positive semidefinite matrix, and
\[
Q_{\alpha\alpha}=2,
\qquad
\operatorname{tr}Q=2p.
\]

\item[(2)]
For $\alpha\ne\beta$,
\[
Q_{\alpha\beta}
=
\begin{cases}
0,
& e_\alpha\cap e_\beta=\varnothing,\\
\pm1,
& |e_\alpha\cap e_\beta|=1.
\end{cases}
\]

\item[(3)]
Changing the orientations of the deleted edges does not change the
eigenvalues of $Q$.
\end{enumerate}
\end{lemma}

\begin{proof}
First, since $Q=B^TB$, we have
\[
Q^T
=
(B^TB)^T
=
B^TB
=
Q.
\]
Thus $Q$ is a real symmetric matrix.

Also, for any $x\in\mathbb{R}^p$, we have
\[
x^TQx
=
x^TB^TBx
=
\|Bx\|^2
\ge0.
\]
Hence $Q$ is positive semidefinite.

Each column vector $b_\alpha$ is the incidence vector obtained by orienting
the deleted edge $e_\alpha$.
It has exactly one entry equal to $1$, one entry equal to $-1$,
and all other entries equal to $0$.
Therefore,
\[
Q_{\alpha\alpha}
=
b_\alpha^Tb_\alpha
=
1^2+(-1)^2
=
2.
\]
Hence
\[
\operatorname{tr}Q
=
\sum_{\alpha=1}^{p}Q_{\alpha\alpha}
=
2p.
\]
This proves (1).

Next, let $\alpha\ne\beta$.
Then
\[
Q_{\alpha\beta}
=
b_\alpha^Tb_\beta.
\]
If the deleted edges $e_\alpha$ and $e_\beta$ do not share an end vertex,
then the nonzero entries of $b_\alpha$ and $b_\beta$ occur in different
positions.
Therefore,
\[
b_\alpha^Tb_\beta=0.
\]

On the other hand, if $e_\alpha$ and $e_\beta$ share exactly one end vertex,
then the nonzero entries of $b_\alpha$ and $b_\beta$ occur in the same
position only at the shared vertex.
At this position, each entry is either $1$ or $-1$.
Hence
\[
b_\alpha^Tb_\beta=\pm1.
\]
The sign depends on the orientations of the two deleted edges.
Therefore,
\[
Q_{\alpha\beta}
=
\begin{cases}
0,
& e_\alpha\cap e_\beta=\varnothing,\\
\pm1,
& |e_\alpha\cap e_\beta|=1.
\end{cases}
\]
This proves (2).

Finally, we prove (3).
If the orientation of a deleted edge $e_\alpha$ is reversed,
then the corresponding column vector $b_\alpha$ is replaced by $-b_\alpha$.
Therefore, if we change the orientations of several deleted edges,
then there exists a diagonal sign matrix
\[
S
=
\operatorname{diag}
(\varepsilon_1,\ldots,\varepsilon_p),
\qquad
\varepsilon_\alpha\in\{1,-1\},
\]
such that
\[
B'=BS.
\]
The new edge-defect matrix is then
\[
\begin{aligned}
Q'
&=
(B')^TB'\\
&=
S^TB^TBS.
\end{aligned}
\]
Since $S$ is a diagonal matrix and satisfies
\[
S^T=S,
\qquad
S^{-1}=S,
\]
we have
\[
Q'
=
SQS
=
S^{-1}QS.
\]
Thus $Q'$ and $Q$ are similar matrices.
Therefore, their eigenvalues are the same, including multiplicities.
\end{proof}

%%%%%%%%%%%%%%%%%%%%%%%%%%%%%%%%%%%%%%%%
\begin{lemma}
\label{lem:laplacian-defect}
Let $J_n$ be the $n\times n$ all-one matrix.
Then
\[
L(K_n-F)
=
nI_n-J_n-BB^T
\]
holds.
\end{lemma}

\begin{proof}
The Laplacian matrix of the complete graph $K_n$ is
\[
\begin{aligned}
L(K_n)
&=nI_n-J_n.
\end{aligned}
\]

For a deleted edge
\[
e_\alpha=\{u_\alpha,v_\alpha\},
\]
put
\[
b_\alpha
=
e_{u_\alpha}-e_{v_\alpha}.
\]
Then
\[
b_\alpha b_\alpha^T
\]
is the Laplacian matrix of the edge $e_\alpha$.
Indeed, this matrix has $1$ in the $(u_\alpha,u_\alpha)$ and
$(v_\alpha,v_\alpha)$ entries, $-1$ in the $(u_\alpha,v_\alpha)$ and
$(v_\alpha,u_\alpha)$ entries, and $0$ in all other entries.

Therefore, deleting
\[
F=\{e_1,\ldots,e_p\}
\]
from $K_n$ gives
\[
L(K_n-F)
=
L(K_n)
-
\sum_{\alpha=1}^{p}
b_\alpha b_\alpha^T.
\]
Since
\[
BB^T
=
\sum_{\alpha=1}^{p}
b_\alpha b_\alpha^T,
\]
we obtain
\[
L(K_n-F)
=
nI_n-J_n-BB^T.
\]
\end{proof}

%%%%%%%%%%%%%%%%%%%%%%%%%%%%%%%%%%%%%%%%
\begin{lemma}
\label{lem:positive-defect-resolvent}
If $K_n-F$ is connected, then
\[
nI_p-Q
\]
is positive definite.
In particular, every eigenvalue $\theta$ of $Q$ satisfies
\[
0\le\theta<n.
\]
\end{lemma}

\begin{proof}
By Lemma \ref{lem:basic-Q} (1), $Q$ is a real symmetric positive
semidefinite matrix.
Thus all eigenvalues of $Q$ are nonnegative.

Let $\theta$ be an eigenvalue of $Q$.
If $\theta=0$, then
\[
n-\theta=n>0.
\]

We now assume that $\theta>0$.
By Lemma \ref{lem:nonzero-spectrum}, $\theta$ is also an eigenvalue
of $BB^T$.
Hence there exists
\[
z\in\mathbb{R}^n\setminus\{\boldsymbol{0}\}
\]
such that
\[
BB^Tz=\theta z.
\]

Since $\theta>0$, we have
\[
z
=
\frac{1}{\theta}BB^Tz
=
\frac{1}{\theta}B(B^Tz).
\]
Thus
\[
z\in\operatorname{Im}B.
\]

On the other hand, each column $b_\alpha$ of $B$ has one entry equal to
$1$ and one entry equal to $-1$.
Therefore,
\[
\boldsymbol{1}^Tb_\alpha=0.
\]
Hence
\[
\operatorname{Im}B
\subseteq
\boldsymbol{1}^{\perp}.
\]
Thus
\[
z\in\boldsymbol{1}^{\perp},
\qquad
J_nz=\boldsymbol{0}.
\]

By Lemma \ref{lem:laplacian-defect},
\[
L(K_n-F)
=
nI_n-J_n-BB^T.
\]
Therefore,
\[
\begin{aligned}
L(K_n-F)z
&=
(nI_n-J_n-BB^T)z\\
&=
nz-\theta z\\
&=
(n-\theta)z.
\end{aligned}
\]
Hence
\[
z^TL(K_n-F)z
=
(n-\theta)\|z\|^2.
\]

Since $K_n-F$ is connected, its Laplacian matrix is positive definite
on $\boldsymbol{1}^{\perp}$.
Now
\[
z\in
\boldsymbol{1}^{\perp}
\setminus\{\boldsymbol{0}\}.
\]
Thus
\[
z^TL(K_n-F)z>0.
\]
Therefore,
\[
(n-\theta)\|z\|^2>0.
\]
Since $z\ne\boldsymbol{0}$, we obtain
\[
n-\theta>0,
\]
that is,
\[
\theta<n.
\]

Thus every eigenvalue $\theta$ of $Q$ satisfies
\[
0\le\theta<n.
\]
Therefore, every eigenvalue
\[
n-\theta
\]
of $nI_p-Q$ is positive.
Hence
\[
nI_p-Q
\]
is positive definite.
\end{proof}

%%%%%%%%%%%%%%%%%%%%%%%%%%%%%%%%%%%%%%%%
\begin{proposition}
\label{prop:matching-characterization}
The following conditions are equivalent:
\[
F\text{ is a matching},
\qquad
H_F\cong pK_2,
\qquad
Q=2I_p.
\]
\end{proposition}

\begin{proof}
First, suppose that $F$ is a matching.
Then any two distinct deleted edges do not share an end vertex.
Therefore, by Lemma \ref{lem:basic-Q} (2), we have
\[
Q_{\alpha\beta}=0
\qquad
(\alpha\ne\beta).
\]
Also, by Lemma \ref{lem:basic-Q} (1), we have
\[
Q_{\alpha\alpha}=2.
\]
Hence
\[
Q=2I_p.
\]

Conversely, suppose that
\[
Q=2I_p.
\]
Then all off-diagonal entries of $Q$ are $0$.
By Lemma \ref{lem:basic-Q} (2), if two distinct deleted edges share
an end vertex, then the corresponding off-diagonal entry is $\pm1$.
Therefore, any two distinct deleted edges do not share an end vertex.
Thus $F$ is a matching.

Moreover, the condition that $F$ consists of $p$ pairwise disjoint edges
is equivalent to saying that the deletion graph $H_F$ is the disjoint union
of $p$ independent edges, that is,
\[
H_F\cong pK_2.
\]
Therefore,
\[
F\text{ is a matching}
\quad\Longleftrightarrow\quad
H_F\cong pK_2
\quad\Longleftrightarrow\quad
Q=2I_p.
\]
\end{proof}

The next lemma shows that the spectrum of the edge-defect matrix is
described by the Laplacian spectrum of the deletion graph.

%%%%%%%%%%%%%%%%%%%%%%%%%%%%%%%%%%%%%%%%
\begin{lemma}
\label{lem:deletion-spectrum}
Suppose that $H_F$ has $q$ vertices, $p$ edges, and $c$ connected components.
Let
\[
\lambda_1,\ldots,\lambda_{q-c}
\]
be the positive Laplacian eigenvalues of $H_F$, counted with multiplicities.
Then the eigenvalues of $Q$ are
\[
\lambda_1,\ldots,\lambda_{q-c}
\]
and
\[
\underbrace{0,\ldots,0}_{p-q+c\text{ times}}.
\]
\end{lemma}

\begin{proof}
First, we check the relation between $BB^T$ and the Laplacian matrix
of $H_F$.
The columns of $B$ are the incidence vectors
\[
b_1,\ldots,b_p
\]
obtained by orienting the deleted edges
\[
F=\{e_1,\ldots,e_p\}.
\]
Therefore,
\[
BB^T
=
\sum_{\alpha=1}^{p}b_\alpha b_\alpha^T.
\]
Each matrix $b_\alpha b_\alpha^T$ is the Laplacian matrix of the edge
$e_\alpha$.
Hence this sum coincides with the Laplacian matrix of $H_F$.
Here, however, $BB^T$ is defined as an $n\times n$ matrix on the vertex set
of $K_n$.
The rows and columns corresponding to the remaining $n-q$ vertices which
do not belong to $H_F$ are all zero.
Thus, after a suitable permutation of the vertices, we have
\[
BB^T
\sim
\begin{bmatrix}
L(H_F) & O\\
O & O
\end{bmatrix}.
\]
Here $\sim$ means similarity by a permutation matrix.

Since $H_F$ has $q$ vertices and $c$ connected components, the multiplicity
of the zero eigenvalue of $L(H_F)$ is $c$.
Therefore, its positive eigenvalues are
\[
\lambda_1,\ldots,\lambda_{q-c},
\]
counted with multiplicities.
By the above block form, the nonzero eigenvalues of $BB^T$ are also
\[
\lambda_1,\ldots,\lambda_{q-c}.
\]

Next, by Lemma \ref{lem:nonzero-spectrum}, the nonzero eigenvalues of
$BB^T$ and
\[
Q=B^TB
\]
are the same, counted with multiplicities.
Therefore, the nonzero eigenvalues of $Q$ are
\[
\lambda_1,\ldots,\lambda_{q-c}.
\]

It remains to determine the multiplicity of the zero eigenvalue of $Q$.
Let $\widehat{B}$ be the matrix obtained from $B$ by taking only the $q$
rows corresponding to the vertices of $H_F$.
Then $\widehat{B}$ is the incidence matrix of the oriented graph $H_F$.

For a vector $x\in\mathbb{R}^q$, the condition
\[
\widehat{B}^Tx=\boldsymbol{0}
\]
is equivalent to
\[
x_u-x_v=0
\]
for every edge $uv\in F$.
Thus $x$ is constant on each connected component of $H_F$.
Hence
\[
\dim\ker \widehat{B}^T=c.
\]
By the rank-nullity theorem,
\[
\operatorname{rank}\widehat{B}^T
=
q-c.
\]
Therefore,
\[
\operatorname{rank}B
=
\operatorname{rank}\widehat{B}
=
q-c.
\]

Also,
\[
\ker(B^TB)=\ker B.
\]
Indeed, if
\[
B^TBy=\boldsymbol{0},
\]
then
\[
0
=
y^TB^TBy
=
\|By\|^2.
\]
Hence $By=\boldsymbol{0}$, and the converse is clear.
Therefore,
\[
\operatorname{rank}Q
=
\operatorname{rank}(B^TB)
=
\operatorname{rank}B
=
q-c.
\]
Since $Q$ is a $p\times p$ matrix, the multiplicity of its zero eigenvalue is
\[
p-\operatorname{rank}Q
=
p-(q-c)
=
p-q+c.
\]

Thus the eigenvalues of $Q$ are
\[
\lambda_1,\ldots,\lambda_{q-c}
\]
and
\[
\underbrace{0,\ldots,0}_{p-q+c\text{ times}}.
\]
\end{proof}

%%%%%%%%%%%%%%%%%%%%%%%%%%%%%%%%%%%%%%%%
\subsection{Majorization}
%%%%%%%%%%%%%%%%%%%%%%%%%%%%%%%%%%%%%%%%

For a real vector
\[
x=(x_1,\ldots,x_p),
\]
we denote by
\[
x^\downarrow
=
(x_1^\downarrow,\ldots,x_p^\downarrow)
\]
the vector obtained by arranging the components of $x$ in decreasing order.

%%%%%%%%%%%%%%%%%%%%%%%%%%%%%%%%%%%%%%%%
\begin{definition}
\label{def:majorization}
For $x,y\in\mathbb{R}^p$, we write
\[
x\succ y
\]
if
\[
\sum_{i=1}^{k}x_i^\downarrow
\ge
\sum_{i=1}^{k}y_i^\downarrow
\qquad
(1\le k\le p-1)
\]
and
\[
\sum_{i=1}^{p}x_i^\downarrow
=
\sum_{i=1}^{p}y_i^\downarrow
\]
hold.
\end{definition}

We use the following standard results
\cite{MarshallOlkinArnold2011,Bhatia1997}.

%%%%%%%%%%%%%%%%%%%%%%%%%%%%%%%%%%%%%%%%
\begin{theorem}
[Karamata's inequality
{\cite[Theorem 1.A.3]{MarshallOlkinArnold2011}}]
\label{thm:karamata}
Let $I\subseteq\mathbb{R}$ be an interval.
Suppose that $x,y\in I^p$ satisfy $x\succ y$.
If $f\colon I\to\mathbb{R}$ is a convex function, then
\[
\sum_{i=1}^{p}f(x_i)
\ge
\sum_{i=1}^{p}f(y_i).
\]
Moreover, if $f$ is strictly convex and $x$ and $y$ are not the same
up to permutation, then the inequality is strict.
\end{theorem}

%%%%%%%%%%%%%%%%%%%%%%%%%%%%%%%%%%%%%%%%
\begin{theorem}
[Cauchy's interlacing theorem
{\cite[Corollary III.1.5]{Bhatia1997}}]
\label{thm:cauchy-interlacing}
Let $A$ be a $p\times p$ real symmetric matrix, and let its eigenvalues be
\[
\lambda_1\ge\cdots\ge\lambda_p.
\]
Let $A'$ be a $q\times q$ principal submatrix of $A$, and let its eigenvalues be
\[
\mu_1\ge\cdots\ge\mu_q.
\]
Then
\[
\lambda_i
\ge
\mu_i
\ge
\lambda_{i+p-q}
\qquad
(1\le i\le q)
\]
holds.
\end{theorem}

%%%%%%%%%%%%%%%%%%%%%%%%%%%%%%%%%%%%%%%%
\section{Effective resistance formula via the edge-defect matrix}
\label{sec:resistance-formula}
%%%%%%%%%%%%%%%%%%%%%%%%%%%%%%%%%%%%%%%%

We use the notation of the previous section.
Thus, let
\[
F=\{e_1,\ldots,e_p\}\subseteq E(K_n),
\qquad
G=K_n-F,
\]
and assume that $G$ is connected.
Let $B$ be the incidence matrix of $F$, and let
\[
Q=B^TB
\]
be the corresponding edge-defect matrix.
For two distinct vertices $x,y\in V(G)$, put
\[
\eta_{xy}=B^T(e_x-e_y).
\]

%%%%%%%%%%%%%%%%%%%%%%%%%%%%%%%%%%%%%%%%
\begin{theorem}
\label{thm:edge-defect-resolvent}
Assume that $G=K_n-F$ is connected.
For two distinct vertices $x,y\in V(G)$, put
\[
\eta_{xy}
=
B^T(e_x-e_y).
\]
Then
\[
R_G(x,y)
=
\frac{2}{n}
+
\frac{1}{n}
\eta_{xy}^T
(nI_p-Q)^{-1}
\eta_{xy}.
\]
\end{theorem}

\begin{proof}
Put
\[
v=e_x-e_y.
\]
Then
\[
\boldsymbol{1}^Tv=0,
\]
and hence
\[
v\in\boldsymbol{1}^{\perp}.
\]
By the standard expression of effective resistance in terms of the
Moore--Penrose inverse, we have
\[
R_G(x,y)
=
v^TL(G)^+v.
\]

By Lemma \ref{lem:laplacian-defect},
\[
L(G)
=
nI_n-J_n-BB^T.
\]
For any $z\in\boldsymbol{1}^{\perp}$, we have
\[
J_nz=\boldsymbol{0}.
\]
Therefore,
\[
L(G)z
=
(nI_n-BB^T)z.
\]
Thus $L(G)$ and $nI_n-BB^T$ coincide on
$\boldsymbol{1}^{\perp}$.

Since $G$ is connected, the kernel of $L(G)$ is
\[
\ker L(G)
=
\operatorname{span}\{\boldsymbol{1}\},
\]
and $L(G)$ is positive definite on $\boldsymbol{1}^{\perp}$.
Moreover, each column of $B$ is orthogonal to $\boldsymbol{1}$, and so
\[
B^T\boldsymbol{1}=\boldsymbol{0}.
\]
Hence
\[
(nI_n-BB^T)\boldsymbol{1}
=
n\boldsymbol{1}.
\]
It follows that $nI_n-BB^T$ is nonsingular on $\mathbb{R}^n$, and for
\(v\in\boldsymbol{1}^{\perp}$, the actions of $L(G)^+$ and
$(nI_n-BB^T)^{-1}$ agree. Hence
\[
L(G)^+v
=
(nI_n-BB^T)^{-1}v.
\]
Therefore,
\[
R_G(x,y)
=
v^T(nI_n-BB^T)^{-1}v.
\]

By Lemma \ref{lem:positive-defect-resolvent}, the matrix
\[
nI_p-Q
\]
is positive definite, and hence nonsingular.
Thus Lemma \ref{lem:resolvent-reduction} gives
\[
(nI_n-BB^T)^{-1}
=
\frac{1}{n}I_n
+
\frac{1}{n}
B(nI_p-Q)^{-1}B^T.
\]
Substituting this into the above expression, we obtain
\[
\begin{aligned}
R_G(x,y)
&=
v^T
\left(
\frac{1}{n}I_n
+
\frac{1}{n}
B(nI_p-Q)^{-1}B^T
\right)
v\\
&=
\frac{1}{n}v^Tv
+
\frac{1}{n}
(B^Tv)^T
(nI_p-Q)^{-1}
(B^Tv).
\end{aligned}
\]
Since $x\ne y$, we have
\[
v^Tv
=
(e_x-e_y)^T(e_x-e_y)
=
2.
\]
Also, by the definition of $\eta_{xy}$,
\[
B^Tv=\eta_{xy}.
\]
Therefore,
\[
R_G(x,y)
=
\frac{2}{n}
+
\frac{1}{n}
\eta_{xy}^T
(nI_p-Q)^{-1}
\eta_{xy}.
\]
\end{proof}

The following corollary gives the spectral form of the effective resistance.

%%%%%%%%%%%%%%%%%%%%%%%%%%%%%%%%%%%%%%%%
\begin{corollary}
\label{thm:spectral-form}
Let
\[
\theta_1,\ldots,\theta_p
\]
be the eigenvalues of $Q$, and let
\[
u_1,\ldots,u_p
\]
be corresponding orthonormal eigenvectors.
Then, for two distinct vertices $x,y\in V(G)$,
\[
R_G(x,y)
=
\frac{2}{n}
+
\frac{1}{n}
\sum_{\alpha=1}^{p}
\frac{
\left(u_\alpha^T\eta_{xy}\right)^2
}{
n-\theta_\alpha
}.
\]
\end{corollary}

\begin{proof}
By Lemma \ref{lem:positive-defect-resolvent}, we have
\[
n-\theta_\alpha>0
\qquad
(1\le\alpha\le p).
\]
Since $Q$ is real symmetric, it has an orthonormal eigenbasis
$u_1,\ldots,u_p$. Thus
\[
Q
=
\sum_{\alpha=1}^{p}
\theta_\alpha u_\alpha u_\alpha^T
\]
and
\[
I_p
=
\sum_{\alpha=1}^{p}
u_\alpha u_\alpha^T.
\]
Therefore,
\[
nI_p-Q
=
\sum_{\alpha=1}^{p}
(n-\theta_\alpha)u_\alpha u_\alpha^T,
\]
and hence
\[
(nI_p-Q)^{-1}
=
\sum_{\alpha=1}^{p}
\frac{1}{n-\theta_\alpha}
u_\alpha u_\alpha^T.
\]
By Theorem \ref{thm:edge-defect-resolvent},
\[
R_G(x,y)
=
\frac{2}{n}
+
\frac{1}{n}
\eta_{xy}^T
(nI_p-Q)^{-1}
\eta_{xy}.
\]
Substituting the above spectral decomposition, we obtain
\[
\begin{aligned}
R_G(x,y)
&=
\frac{2}{n}
+
\frac{1}{n}
\sum_{\alpha=1}^{p}
\frac{1}{n-\theta_\alpha}
\eta_{xy}^Tu_\alpha u_\alpha^T\eta_{xy}\\
&=
\frac{2}{n}
+
\frac{1}{n}
\sum_{\alpha=1}^{p}
\frac{
\left(u_\alpha^T\eta_{xy}\right)^2
}{
n-\theta_\alpha
}.
\end{aligned}
\]
This proves the corollary.
\end{proof}

%%%%%%%%%%%%%%%%%%%%%%%%%%%%%%%%%%%%%%%%
\section{Kirchhoff index and the number of spanning trees}
\label{sec:kirchhoff-spanning}
%%%%%%%%%%%%%%%%%%%%%%%%%%%%%%%%%%%%%%%%

Assume that $G=K_n-F$ is connected.
Let $B$ be the incidence matrix of $F$, and let
\[
Q=B^TB
\]
be the edge-defect matrix.

%%%%%%%%%%%%%%%%%%%%%%%%%%%%%%%%%%%%%%%%
\begin{theorem}
\label{thm:kirchhoff-spanning-defect}
Let
\[
\theta_1,\ldots,\theta_p
\]
be the eigenvalues of $Q$, including zero eigenvalues.
Then
\[
\begin{aligned}
Kf(K_n-F)
&=
n-1+
\operatorname{tr}
\left(
Q(nI_p-Q)^{-1}
\right)\\
&=
n-1+
\sum_{\alpha=1}^{p}
\frac{\theta_\alpha}{n-\theta_\alpha}.
\end{aligned}
\]
\end{theorem}

\begin{proof}
Let
\[
\theta_1,\ldots,\theta_r
\]
be the nonzero eigenvalues of $Q$.
By Lemma \ref{lem:laplacian-defect}, we have
\[
L(K_n-F)=nI_n-J_n-BB^T.
\]
Since
\[
\operatorname{Im}B\subseteq\boldsymbol{1}^{\perp},
\]
the restriction of $L(K_n-F)$ to $\boldsymbol{1}^{\perp}$ is equal to
the restriction of
\[
nI_n-BB^T
\]
to $\boldsymbol{1}^{\perp}$.

By Lemma \ref{lem:nonzero-spectrum}, the nonzero eigenvalues of $BB^T$
and
\[
Q=B^TB
\]
are the same, counted with multiplicities.
Therefore, the nonzero Laplacian eigenvalues of $K_n-F$ are
\[
n-\theta_1,\ldots,n-\theta_r
\]
and
\[
\underbrace{n,\ldots,n}_{n-1-r\text{ times}}.
\]
Hence, by Theorem \ref{thm:kirchhoff-spectral},
\[
\begin{aligned}
Kf(K_n-F)
&=
n\left(
\sum_{\alpha=1}^{r}
\frac{1}{n-\theta_\alpha}
+
\frac{n-1-r}{n}
\right)\\
&=
\sum_{\alpha=1}^{r}
\frac{n}{n-\theta_\alpha}
+
n-1-r.
\end{aligned}
\]
Since
\[
\frac{n}{n-\theta_\alpha}
=
1+
\frac{\theta_\alpha}{n-\theta_\alpha},
\]
we obtain
\[
\begin{aligned}
Kf(K_n-F)
&=
\sum_{\alpha=1}^{r}
\left(
1+
\frac{\theta_\alpha}{n-\theta_\alpha}
\right)
+
n-1-r\\
&=
n-1+
\sum_{\alpha=1}^{r}
\frac{\theta_\alpha}{n-\theta_\alpha}.
\end{aligned}
\]
The remaining eigenvalues of $Q$ are zero, and so
\[
\frac{\theta_\alpha}{n-\theta_\alpha}=0
\qquad
(r+1\le\alpha\le p).
\]
Therefore,
\[
Kf(K_n-F)
=
n-1+
\sum_{\alpha=1}^{p}
\frac{\theta_\alpha}{n-\theta_\alpha}.
\]

Finally, since $Q$ is real symmetric, it can be orthogonally diagonalized.
Thus the eigenvalues of
\[
Q(nI_p-Q)^{-1}
\]
are
\[
\frac{\theta_1}{n-\theta_1},\ldots,
\frac{\theta_p}{n-\theta_p}.
\]
Taking the trace, we get
\[
\operatorname{tr}
\left(
Q(nI_p-Q)^{-1}
\right)
=
\sum_{\alpha=1}^{p}
\frac{\theta_\alpha}{n-\theta_\alpha}.
\]
This proves the theorem.
\end{proof}

%%%%%%%%%%%%%%%%%%%%%%%%%%%%%%%%%%%%%%%%
\begin{theorem}
\label{thm:spanning-tree-defect}
Let
\[
\theta_1,\ldots,\theta_p
\]
be the eigenvalues of $Q$, including zero eigenvalues.
Then
\[
\begin{aligned}
\tau(K_n-F)
&=
n^{n-2}
\prod_{\alpha=1}^{p}
\left(
1-\frac{\theta_\alpha}{n}
\right)\\
&=
n^{n-p-2}\det(nI_p-Q).
\end{aligned}
\]
\end{theorem}

\begin{proof}
Let
\[
\theta_1,\ldots,\theta_r
\]
be the nonzero eigenvalues of $Q$.
As in the proof of Theorem \ref{thm:kirchhoff-spanning-defect},
the nonzero Laplacian eigenvalues of $K_n-F$ are
\[
n-\theta_1,\ldots,n-\theta_r
\]
and
\[
\underbrace{n,\ldots,n}_{n-1-r\text{ times}}.
\]
Therefore, by the Matrix--Tree Theorem,
\[
\begin{aligned}
\tau(K_n-F)
&=
\frac{1}{n}
\left(
\prod_{\alpha=1}^{r}(n-\theta_\alpha)
\right)
n^{n-1-r}\\
&=
\frac{1}{n}
\left(
\prod_{\alpha=1}^{r}
n\left(
1-\frac{\theta_\alpha}{n}
\right)
\right)
n^{n-1-r}\\
&=
n^{n-2}
\prod_{\alpha=1}^{r}
\left(
1-\frac{\theta_\alpha}{n}
\right).
\end{aligned}
\]
Since
\[
\theta_{r+1}=\cdots=\theta_p=0,
\]
we have
\[
1-\frac{\theta_\alpha}{n}=1
\qquad
(r+1\le\alpha\le p).
\]
Hence
\[
\tau(K_n-F)
=
n^{n-2}
\prod_{\alpha=1}^{p}
\left(
1-\frac{\theta_\alpha}{n}
\right).
\]

On the other hand, the eigenvalues of $nI_p-Q$ are
\[
n-\theta_1,\ldots,n-\theta_p.
\]
Thus
\[
\begin{aligned}
\det(nI_p-Q)
&=
\prod_{\alpha=1}^{p}(n-\theta_\alpha)\\
&=
n^p
\prod_{\alpha=1}^{p}
\left(
1-\frac{\theta_\alpha}{n}
\right).
\end{aligned}
\]
Therefore,
\[
\prod_{\alpha=1}^{p}
\left(
1-\frac{\theta_\alpha}{n}
\right)
=
n^{-p}\det(nI_p-Q).
\]
Substituting this into the product formula for $\tau(K_n-F)$, we obtain
\[
\begin{aligned}
\tau(K_n-F)
&=
n^{n-2}n^{-p}\det(nI_p-Q)\\
&=
n^{n-p-2}\det(nI_p-Q).
\end{aligned}
\]
This completes the proof.
\end{proof}

%%%%%%%%%%%%%%%%%%%%%%%%%%%%%%%%%%%%%%%%
\section{Stability identity and a refinement of the Xu--Das--Zhang bound}
\label{sec:stability}
%%%%%%%%%%%%%%%%%%%%%%%%%%%%%%%%%%%%%%%%

In this section, we use the spectrum of the edge-defect matrix to describe
the excess from the Xu--Das--Zhang type lower bound.

%%%%%%%%%%%%%%%%%%%%%%%%%%%%%%%%%%%%%%%%
\begin{theorem}
\label{thm:exact-stability}
Let $n\ge3$, and assume that $K_n-F$ is connected.
Then
\[
\begin{aligned}
Kf(K_n-F)
&=
n-1+\frac{2p}{n-2}+\frac{n}{(n-2)^2}\operatorname{tr}
\left((Q-2I_p)^2(nI_p-Q)^{-1}\right).
\end{aligned}
\]
Equivalently,
\[
Kf(K_n-F)
=n-1+\frac{2p}{n-2}+\frac{n}{(n-2)^2}
\sum_{\alpha=1}^{p}\frac{(\theta_\alpha-2)^2}{n-\theta_\alpha}.
\]
\end{theorem}

\begin{proof}
By Theorem \ref{thm:kirchhoff-spanning-defect}, we have
\[
Kf(K_n-F)
=
n-1+
\sum_{\alpha=1}^{p}
\frac{\theta_\alpha}{n-\theta_\alpha}.
\]
Thus it is enough to rewrite each term
\[
\frac{\theta_\alpha}{n-\theta_\alpha}.
\]
By Lemma \ref{lem:positive-defect-resolvent}, every eigenvalue of $Q$
satisfies
\[
\theta_\alpha<n.
\]
Hence all denominators below are nonzero.

For any real number $t<n$, we have the identity
\[
\frac{t}{n-t}
=
\frac{2}{n-2}
+
\frac{n(t-2)}{(n-2)^2}
+
\frac{n(t-2)^2}{(n-2)^2(n-t)}.
\]
Indeed, after taking a common denominator, the right-hand side becomes
\[
\begin{aligned}
&\frac{
2(n-2)(n-t)
+
n(t-2)(n-t)
+
n(t-2)^2
}{
(n-2)^2(n-t)
}
\\
&=
\frac{
2(n-2)(n-t)
+
n(t-2)\{(n-t)+(t-2)\}
}{
(n-2)^2(n-t)
}
\\
&=
\frac{
2(n-2)(n-t)
+
n(t-2)(n-2)
}{
(n-2)^2(n-t)
}
\\
&=
\frac{
(n-2)\{2(n-t)+n(t-2)\}
}{
(n-2)^2(n-t)
}
=
\frac{t}{n-t}.
\end{aligned}
\]
Applying this identity to $t=\theta_\alpha$, we obtain
\[
\frac{\theta_\alpha}{n-\theta_\alpha}
=
\frac{2}{n-2}
+
\frac{n(\theta_\alpha-2)}{(n-2)^2}
+
\frac{n(\theta_\alpha-2)^2}{(n-2)^2(n-\theta_\alpha)}.
\]
Summing over $\alpha=1,\ldots,p$, we get
\[
\begin{aligned}
\sum_{\alpha=1}^{p}
\frac{\theta_\alpha}{n-\theta_\alpha}
&=
\frac{2p}{n-2}
+
\frac{n}{(n-2)^2}
\sum_{\alpha=1}^{p}(\theta_\alpha-2)
\\
&\quad+
\frac{n}{(n-2)^2}
\sum_{\alpha=1}^{p}
\frac{(\theta_\alpha-2)^2}{n-\theta_\alpha}.
\end{aligned}
\]
Since the sum of the eigenvalues of $Q$ is equal to $\operatorname{tr}Q$,
and since $\operatorname{tr}Q=2p$ by Lemma \ref{lem:basic-Q} (1), we have
\[
\sum_{\alpha=1}^{p}(\theta_\alpha-2)
=
\operatorname{tr}Q-2p
=
0.
\]
Therefore,
\[
\sum_{\alpha=1}^{p}
\frac{\theta_\alpha}{n-\theta_\alpha}
=
\frac{2p}{n-2}
+
\frac{n}{(n-2)^2}
\sum_{\alpha=1}^{p}
\frac{(\theta_\alpha-2)^2}{n-\theta_\alpha}.
\]
Substituting this into the formula of Theorem
\ref{thm:kirchhoff-spanning-defect}, we obtain
\[
Kf(K_n-F)
=
n-1+\frac{2p}{n-2}
+
\frac{n}{(n-2)^2}
\sum_{\alpha=1}^{p}
\frac{(\theta_\alpha-2)^2}{n-\theta_\alpha}.
\]

It remains to prove the trace form.
Since $Q$ is real symmetric, it can be orthogonally diagonalized as
\[
Q
=
U
\operatorname{diag}
(\theta_1,\ldots,\theta_p)
U^T.
\]
Then
\[
Q-2I_p
=
U
\operatorname{diag}
(\theta_1-2,\ldots,\theta_p-2)
U^T
\]
and
\[
(nI_p-Q)^{-1}
=
U
\operatorname{diag}
\left(
\frac{1}{n-\theta_1},
\ldots,
\frac{1}{n-\theta_p}
\right)
U^T.
\]
Thus
\[
\begin{aligned}
&(Q-2I_p)^2(nI_p-Q)^{-1}
\\
&=
U
\operatorname{diag}
\left(
\frac{(\theta_1-2)^2}{n-\theta_1},
\ldots,
\frac{(\theta_p-2)^2}{n-\theta_p}
\right)
U^T.
\end{aligned}
\]
Taking the trace gives
\[
\operatorname{tr}
\left(
(Q-2I_p)^2(nI_p-Q)^{-1}
\right)
=
\sum_{\alpha=1}^{p}
\frac{(\theta_\alpha-2)^2}{n-\theta_\alpha}.
\]
Hence the two expressions are equivalent.
\end{proof}

%%%%%%%%%%%%%%%%%%%%%%%%%%%%%%%%%%%%%%%%
\begin{corollary}
\label{cor:xu-das-zhang-recovery}
Let $n\ge3$, and assume that $K_n-F$ is connected.
Then
\[
Kf(K_n-F)
\ge
n-1+\frac{2p}{n-2}.
\]
Equality holds if and only if
\[
H_F\cong pK_2.
\]
\end{corollary}

\begin{proof}
By Lemma \ref{lem:positive-defect-resolvent}, we have
\[
n-\theta_\alpha>0
\qquad
(1\le\alpha\le p).
\]
Hence the correction term in Theorem \ref{thm:exact-stability} is
nonnegative.

Equality holds if and only if
\[
\theta_\alpha=2
\qquad
(1\le\alpha\le p).
\]
Since $Q$ is real symmetric, this is equivalent to
\[
Q=2I_p.
\]
By Proposition \ref{prop:matching-characterization}, this is equivalent to
\[
H_F\cong pK_2.
\]
\end{proof}

Next, we relate local interactions among the deleted edges to the spread
of the eigenvalues of the edge-defect matrix.
Put
\[
a(F)
=
\sum_{v\in V(H_F)}
\binom{d_{H_F}(v)}{2}.
\]
This is the number of unordered pairs of deleted edges which share a vertex.

%%%%%%%%%%%%%%%%%%%%%%%%%%%%%%%%%%%%%%%%
\begin{lemma}
\label{lem:trace-A-square}
Put
\[
A=Q-2I_p.
\]
Then
\[
\operatorname{tr}A^2=2a(F).
\]
\end{lemma}

\begin{proof}
By Lemma \ref{lem:basic-Q} (1), (2), the matrix $A=Q-2I_p$ is real
symmetric, its diagonal entries are all zero, and its off-diagonal entries
are $0$ or $\pm1$.
Moreover, for $\alpha\ne\beta$,
\[
A_{\alpha\beta}\ne0
\]
if and only if the deleted edges $e_\alpha$ and $e_\beta$ share a vertex.

Since $A$ is symmetric,
\[
\begin{aligned}
\operatorname{tr}A^2
&=
\sum_{\alpha=1}^{p}(A^2)_{\alpha\alpha}\\
&=
\sum_{\alpha,\beta=1}^{p}
A_{\alpha\beta}A_{\beta\alpha}\\
&=
\sum_{\alpha,\beta=1}^{p}
A_{\alpha\beta}^2.
\end{aligned}
\]
For each unordered pair $\{e_\alpha,e_\beta\}$ of deleted edges sharing
a vertex, the two terms
\[
A_{\alpha\beta}^2
=
A_{\beta\alpha}^2
=
1
\]
appear in the last sum.
Therefore,
\[
\operatorname{tr}A^2=2a(F).
\]
\end{proof}

In the next section, we use this identity and majorization to determine
the minimum among non-matching deleted edge sets and its equality condition.

%%%%%%%%%%%%%%%%%%%%%%%%%%%%%%%%%%%%%%%%
\section{Majorization and the first stability gap}
\label{sec:first-gap}
%%%%%%%%%%%%%%%%%%%%%%%%%%%%%%%%%%%%%%%%

In this section, we use majorization to determine the structure which
minimizes the Kirchhoff index among non-matching deletion graphs.
We first show that a majorization relation for the eigenvalue vectors of
edge-defect matrices gives an order relation for Kirchhoff indices.

%%%%%%%%%%%%%%%%%%%%%%%%%%%%%%%%%%%%%%%%
\begin{theorem}
\label{thm:majorization-principle}
Let $F_1,F_2\subseteq E(K_n)$ be sets of $p$ edges, and assume that
$K_n-F_1$ and $K_n-F_2$ are connected.
For $i=1,2$, let
\[
\theta(F_i)
=
\bigl(
\theta_1(F_i),\ldots,\theta_p(F_i)
\bigr)
\]
be the eigenvalue vector of the edge-defect matrix of $F_i$, arranged in
decreasing order.
If
\[
\theta(F_1)\succ\theta(F_2),
\]
then
\[
Kf(K_n-F_1)\ge Kf(K_n-F_2).
\]
Moreover, if the two eigenvalue vectors are not the same up to permutation,
then the inequality is strict.
\end{theorem}

\begin{proof}
Let $Q_i$ be the edge-defect matrix corresponding to $F_i$, and write
\[
\theta_1(F_i)\ge\cdots\ge\theta_p(F_i)
\]
for its eigenvalues.

Consider the function
\[
f(t)=\frac{t}{n-t}.
\]
It is twice differentiable on $t<n$, and
\[
f'(t)
=
\frac{n}{(n-t)^2},
\qquad
f''(t)
=
\frac{2n}{(n-t)^3}.
\]
Hence $f$ is strictly convex on $[0,n)$.

Since each $Q_i$ is real symmetric and positive semidefinite, its
eigenvalues are nonnegative. Also, since $K_n-F_i$ is connected,
Lemma \ref{lem:positive-defect-resolvent} gives
\[
\theta_\alpha(F_i)<n
\qquad
(1\le\alpha\le p,\ i=1,2).
\]
Thus all components of the two eigenvalue vectors belong to $[0,n)$.

The assumption
\[
\theta(F_1)\succ\theta(F_2)
\]
allows us to apply Karamata's inequality, Theorem \ref{thm:karamata}, to
the strictly convex function $f$. Therefore,
\[
\sum_{\alpha=1}^{p}
f\bigl(\theta_\alpha(F_1)\bigr)
\ge
\sum_{\alpha=1}^{p}
f\bigl(\theta_\alpha(F_2)\bigr).
\]
That is,
\[
\sum_{\alpha=1}^{p}
\frac{\theta_\alpha(F_1)}
     {n-\theta_\alpha(F_1)}
\ge
\sum_{\alpha=1}^{p}
\frac{\theta_\alpha(F_2)}
     {n-\theta_\alpha(F_2)}.
\]
By Theorem \ref{thm:kirchhoff-spanning-defect},
\[
Kf(K_n-F_i)
=
n-1+
\sum_{\alpha=1}^{p}
\frac{\theta_\alpha(F_i)}
     {n-\theta_\alpha(F_i)}
\qquad
(i=1,2).
\]
Hence
\[
Kf(K_n-F_1)\ge Kf(K_n-F_2).
\]

Since $f$ is strictly convex, the equality condition in Karamata's
inequality shows that the inequality is strict unless the two eigenvalue
vectors are the same up to permutation.
This proves the theorem.
\end{proof}

Next, we show that every non-matching deleted edge set has an eigenvalue
vector which majorizes a common vector.

%%%%%%%%%%%%%%%%%%%%%%%%%%%%%%%%%%%%%%%%
\begin{lemma}
\label{lem:nonmatching-majorization}
Let $p\ge2$, and suppose that $F$ is not a matching.
Let the eigenvalues of the edge-defect matrix $Q$ be arranged as
\[
\theta_1\ge\cdots\ge\theta_p.
\]
Then
\[
(\theta_1,\ldots,\theta_p)
\succ
\left(
3,
\underbrace{2,\ldots,2}_{p-2},
1
\right).
\]
\end{lemma}

\begin{proof}
Put
\[
A=Q-2I_p,
\]
and let its eigenvalues be
\[
\delta_1\ge\cdots\ge\delta_p.
\]
Since $Q=A+2I_p$, we have
\[
\theta_i=2+\delta_i
\qquad
(1\le i\le p).
\]
Also, by Lemma \ref{lem:basic-Q} (1),
\[
\operatorname{tr}Q=2p.
\]
Hence
\[
\sum_{i=1}^{p}\delta_i
=
\operatorname{tr}A
=
\operatorname{tr}Q-2p
=
0.
\]

Since $F$ is not a matching, there exist two deleted edges which share a
vertex. Let them be $e_\alpha$ and $e_\beta$.
By Lemma \ref{lem:basic-Q} (2), the corresponding $2\times2$ principal
submatrix of $A$ is
\[
\begin{bmatrix}
0&\varepsilon\\
\varepsilon&0
\end{bmatrix},
\qquad
\varepsilon\in\{1,-1\}.
\]
Its eigenvalues are
\[
1,\ -1.
\]
Therefore, by Cauchy's interlacing theorem,
\[
\delta_1\ge1,
\qquad
\delta_p\le-1.
\]

We now show that
\[
\sum_{i=1}^{k}\delta_i\ge1
\qquad
(1\le k\le p-1).
\]
Fix $k$. If
\[
\delta_k\ge0,
\]
then
\[
\delta_1,\ldots,\delta_k\ge0,
\]
and so
\[
\sum_{i=1}^{k}\delta_i
\ge
\delta_1
\ge
1.
\]
If
\[
\delta_k<0,
\]
then
\[
\delta_{k+1},\ldots,\delta_p<0.
\]
Since $\delta_p\le -1$, we have
\[
\sum_{i=k+1}^{p}\delta_i
\le
\delta_p
\le
-1.
\]
Using
\[
\sum_{i=1}^{p}\delta_i=0,
\]
we obtain
\[
\sum_{i=1}^{k}\delta_i
=
-\sum_{i=k+1}^{p}\delta_i
\ge
1.
\]

Thus, for every $1\le k\le p-1$,
\[
\sum_{i=1}^{k}\delta_i\ge1.
\]
Also,
\[
\sum_{i=1}^{p}\delta_i=0,
\]
which is the same as the sum of the components of
\[
\left(
1,
\underbrace{0,\ldots,0}_{p-2},
-1
\right).
\]
Hence, by the definition of majorization,
\[
(\delta_1,\ldots,\delta_p)
\succ
\left(
1,
\underbrace{0,\ldots,0}_{p-2},
-1
\right).
\]
Adding $2$ to each component preserves majorization. Therefore,
\[
(\theta_1,\ldots,\theta_p)
=
(\delta_1+2,\ldots,\delta_p+2)
\succ
\left(
3,
\underbrace{2,\ldots,2}_{p-2},
1
\right).
\]
\end{proof}

We now determine the minimum among non-matching deleted edge sets by using
this lemma and Theorem \ref{thm:majorization-principle}.

%%%%%%%%%%%%%%%%%%%%%%%%%%%%%%%%%%%%%%%%
\begin{theorem}
\label{thm:first-gap}
Let $p\ge2$ and
\[
n\ge\max\{4,2p-1\}.
\]
Let $F\subseteq E(K_n)$ be a set of $p$ edges, and assume that
$K_n-F$ is connected.

Let $F_0\subseteq E(K_n)$ be a deleted edge set such that
\[
H_{F_0}\cong P_3\cup(p-2)K_2.
\]
If $H_F$ is not a matching, then
\[
Kf(K_n-F)
\ge
Kf(K_n-F_0).
\]
More explicitly,
\[
Kf(K_n-F)
\ge
n-1
+
\frac{3}{n-3}
+
\frac{1}{n-1}
+
\frac{2(p-2)}{n-2}.
\]
Equality holds if and only if
\[
H_F\cong P_3\cup(p-2)K_2.
\]
\end{theorem}

\begin{proof}
First, we check that the reference deleted edge set $F_0$ can be realized.
The graph
\[
P_3\cup(p-2)K_2
\]
has
\[
3+2(p-2)=2p-1
\]
vertices. Hence, by the assumption
\[
n\ge2p-1,
\]
it can be realized as a subgraph of $K_n$.

Also, since $n\ge4$, the graph $K_n-F_0$ is connected.
Indeed, let $x,y$ be two distinct vertices.
If $\{x,y\}\notin F_0$, then $x$ and $y$ are adjacent in $K_n-F_0$.
If $\{x,y\}\in F_0$, then we can choose a vertex $z$ which is not joined
to either $x$ or $y$ by a deleted edge. Then
\[
xz,\ yz\in E(K_n-F_0),
\]
and so
\[
x-z-y
\]
is a path in $K_n-F_0$.

Next, we compute the eigenvalues of the edge-defect matrix of $F_0$.
The Laplacian eigenvalues of $P_3$ are
\[
0,\ 1,\ 3,
\]
and those of $K_2$ are
\[
0,\ 2.
\]
Since the Laplacian matrix of a disjoint union is the direct sum of the
Laplacian matrices of its connected components, the positive Laplacian
eigenvalues of $H_{F_0}$ are
\[
3,
\underbrace{2,\ldots,2}_{p-2},
1.
\]
By Lemma \ref{lem:deletion-spectrum}, the eigenvalues of the edge-defect
matrix of $F_0$ are also
\[
3,
\underbrace{2,\ldots,2}_{p-2},
1.
\]

On the other hand, since $H_F$ is not a matching, Lemma
\ref{lem:nonmatching-majorization} shows that the eigenvalue vector of the
edge-defect matrix of $F$ majorizes
\[
\left(
3,
\underbrace{2,\ldots,2}_{p-2},
1
\right).
\]
Therefore, by Theorem \ref{thm:majorization-principle},
\[
Kf(K_n-F)
\ge
Kf(K_n-F_0).
\]

By Theorem \ref{thm:kirchhoff-spanning-defect}, we have
\[
\begin{aligned}
Kf(K_n-F_0)
&=
n-1
+
\frac{3}{n-3}
+
\sum_{\alpha=1}^{p-2}\frac{2}{n-2}
+
\frac{1}{n-1}\\
&=
n-1
+
\frac{3}{n-3}
+
\frac{1}{n-1}
+
\frac{2(p-2)}{n-2}.
\end{aligned}
\]
This gives the desired lower bound.

It remains to discuss the equality condition.
Suppose that
\[
Kf(K_n-F)
=
Kf(K_n-F_0).
\]
Since Theorem \ref{thm:majorization-principle} is based on the strictly
convex function
\[
f(t)=\frac{t}{n-t},
\]
the equality condition implies that the eigenvalue vectors of $F$ and
$F_0$ are the same up to permutation.
As the eigenvalues are arranged in decreasing order, the eigenvalues of
$Q$ are
\[
3,
\underbrace{2,\ldots,2}_{p-2},
1.
\]
Thus the eigenvalues of
\[
A=Q-2I_p
\]
are
\[
1,
\underbrace{0,\ldots,0}_{p-2},
-1.
\]
Hence
\[
\operatorname{tr}A^2=2.
\]
By Lemma \ref{lem:trace-A-square},
\[
\operatorname{tr}A^2=2a(F).
\]
Therefore,
\[
a(F)=1.
\]
This means that there is exactly one unordered pair of deleted edges which
share a vertex. These two edges form a copy of $P_3$, and the remaining
\(p-2$ deleted edges share no vertices with each other or with this
$P_3$. Hence
\[
H_F\cong P_3\cup(p-2)K_2.
\]

Conversely, suppose that
\[
H_F\cong P_3\cup(p-2)K_2.
\]
Then, by Lemma \ref{lem:deletion-spectrum}, the eigenvalues of the
edge-defect matrix of $F$ are
\[
3,
\underbrace{2,\ldots,2}_{p-2},
1.
\]
These are the same as those for $F_0$.
Therefore, Theorem \ref{thm:kirchhoff-spanning-defect} gives
\[
Kf(K_n-F)=Kf(K_n-F_0).
\]
This proves the equality condition.
\end{proof}

Assume further that $n\ge2p$, and let $M\subseteq E(K_n)$ be a matching
with $p$ edges.
By the assumption $n\ge2p$, such a matching can be realized in $K_n$.
The difference between the minimum attained by a matching and the minimum
among non-matching deleted edge sets is as follows.

%%%%%%%%%%%%%%%%%%%%%%%%%%%%%%%%%%%%%%%%
\begin{corollary}
\label{cor:first-gap-size}
\[
Kf(K_n-F_0)-Kf(K_n-M)
=
\frac{2n}{(n-1)(n-2)(n-3)}.
\]
\end{corollary}

\begin{proof}
Since $M$ is a matching with $p$ edges, Corollary
\ref{cor:xu-das-zhang-recovery} gives
\[
Kf(K_n-M)
=
n-1+\frac{2p}{n-2}.
\]
On the other hand, by Theorem \ref{thm:first-gap},
\[
Kf(K_n-F_0)
=
n-1
+
\frac{3}{n-3}
+
\frac{1}{n-1}
+
\frac{2(p-2)}{n-2}.
\]
Therefore,
\[
\begin{aligned}
Kf(K_n-F_0)-Kf(K_n-M)
&=
\frac{3}{n-3}
+
\frac{1}{n-1}
+
\frac{2(p-2)}{n-2}
-
\frac{2p}{n-2}\\
&=
\frac{1}{n-1}
+
\frac{3}{n-3}
-
\frac{4}{n-2}.
\end{aligned}
\]
Taking a common denominator, we get
\[
\begin{aligned}
&\frac{1}{n-1}
+
\frac{3}{n-3}
-
\frac{4}{n-2}
\\
&=
\frac{
(n-2)(n-3)
+
3(n-1)(n-2)
-
4(n-1)(n-3)
}{
(n-1)(n-2)(n-3)
}\\
&=
\frac{2n}{(n-1)(n-2)(n-3)}.
\end{aligned}
\]
Hence
\[
Kf(K_n-F_0)-Kf(K_n-M)
=
\frac{2n}{(n-1)(n-2)(n-3)}.
\]
\end{proof}

%%%%%%%%%%%%%%%%%%%%%%%%%%%%%%%%%%%%%%%%
\begin{remark}
\label{rem:first-gap-palacios}
When $p=2$, we have
\[
H_{F_0}\cong P_3,
\]
which is consistent with the result of Palacios \cite{Palacios2015}.
When $p=3$, the graph
\[
H_{F_0}\cong P_3\cup K_2
\]
also appears in the known ranking in Xu, Das and Zhang
\cite{XuDasZhang2016}.
Theorem \ref{thm:first-gap} extends this structure to arbitrary $p$,
and gives the sharp minimum and the unique equality condition among
non-matching deleted edge sets.
\end{remark}

%%%%%%%%%%%%%%%%%%%%%%%%%%%%%%%%%%%%%%%%
\section{Examples}
\label{sec:examples}
%%%%%%%%%%%%%%%%%%%%%%%%%%%%%%%%%%%%%%%%

In this section, we apply the formulas obtained above to some typical
deletion graphs.
Throughout this section, we assume that
\[
G=K_n-F
\]
is connected.

By Lemma \ref{lem:deletion-spectrum}, if
\[
\lambda_1,\ldots,\lambda_r
\]
are the positive Laplacian eigenvalues of $H_F$, counted with
multiplicities, then the nonzero eigenvalues of the edge-defect matrix
$Q$ are also
\[
\lambda_1,\ldots,\lambda_r.
\]
Hence, by Theorems \ref{thm:kirchhoff-spanning-defect} and
\ref{thm:spanning-tree-defect}, we have
\[
Kf(K_n-F)
=
n-1+
\sum_{j=1}^{r}
\frac{\lambda_j}{n-\lambda_j}
\]
and
\[
\tau(K_n-F)
=
n^{n-2}
\prod_{j=1}^{r}
\left(
1-\frac{\lambda_j}{n}
\right).
\]

We first give formulas for some basic deletion graphs.

%%%%%%%%%%%%%%%%%%%%%%%%%%%%%%%%%%%%%%%%
\begin{proposition}
\label{prop:matching-defect}
Suppose that $H_F\cong mK_2$.
Then
\[
Kf(K_n-F)
=
n-1+\frac{2m}{n-2}
\]
and
\[
\tau(K_n-F)
=
n^{n-m-2}(n-2)^m.
\]
\end{proposition}

\begin{proof}
The Laplacian eigenvalues of $K_2$ are
\[
0,\ 2.
\]
Since $mK_2$ is the disjoint union of $m$ copies of $K_2$, the positive
Laplacian eigenvalues of $H_F$ are
\[
\underbrace{2,\ldots,2}_{m\text{ times}}.
\]
Therefore,
\[
\begin{aligned}
Kf(K_n-F)
&=
n-1+
\sum_{\alpha=1}^{m}
\frac{2}{n-2}\\
&=
n-1+\frac{2m}{n-2}.
\end{aligned}
\]
Also,
\[
\begin{aligned}
\tau(K_n-F)
&=
n^{n-2}
\prod_{\alpha=1}^{m}
\left(
1-\frac{2}{n}
\right)\\
&=
n^{n-2}
\left(
\frac{n-2}{n}
\right)^m\\
&=
n^{n-m-2}(n-2)^m.
\end{aligned}
\]
\end{proof}

%%%%%%%%%%%%%%%%%%%%%%%%%%%%%%%%%%%%%%%%
\begin{proposition}
\label{prop:first-nonmatching-defect}
Let $m\ge2$, and suppose that
\[
H_F\cong P_3\cup(m-2)K_2.
\]
Then
\[
Kf(K_n-F)
=
n-1
+
\frac{3}{n-3}
+
\frac{1}{n-1}
+
\frac{2(m-2)}{n-2}
\]
and
\[
\tau(K_n-F)
=
n^{n-m-2}
(n-3)(n-1)(n-2)^{m-2}.
\]
\end{proposition}

\begin{proof}
The Laplacian eigenvalues of $P_3$ are
\[
0,\ 1,\ 3.
\]
Indeed, the Laplacian matrix of $P_3$ is
\[
L(P_3)
=
\begin{bmatrix}
1&-1&0\\
-1&2&-1\\
0&-1&1
\end{bmatrix},
\]
and its characteristic polynomial is
\[
\lambda(\lambda-1)(\lambda-3).
\]
Also, each copy of $K_2$ has Laplacian eigenvalues
\[
0,\ 2.
\]
Thus the positive Laplacian eigenvalues of $H_F$ are
\[
3,\ 1,\
\underbrace{2,\ldots,2}_{m-2\text{ times}}.
\]
Therefore,
\[
\begin{aligned}
Kf(K_n-F)
&=
n-1
+
\frac{3}{n-3}
+
\frac{1}{n-1}
+
\sum_{\alpha=1}^{m-2}
\frac{2}{n-2}\\
&=
n-1
+
\frac{3}{n-3}
+
\frac{1}{n-1}
+
\frac{2(m-2)}{n-2}.
\end{aligned}
\]
Also,
\[
\begin{aligned}
\tau(K_n-F)
&=
n^{n-2}
\left(
1-\frac{3}{n}
\right)
\left(
1-\frac{1}{n}
\right)
\left(
1-\frac{2}{n}
\right)^{m-2}\\
&=
n^{n-2}
\frac{n-3}{n}
\frac{n-1}{n}
\left(
\frac{n-2}{n}
\right)^{m-2}\\
&=
n^{n-m-2}
(n-3)(n-1)(n-2)^{m-2}.
\end{aligned}
\]
\end{proof}

%%%%%%%%%%%%%%%%%%%%%%%%%%%%%%%%%%%%%%%%
\begin{proposition}
\label{prop:star-defect}
Suppose that $H_F\cong K_{1,m}$ and
\[
n\ge m+2.
\]
Then
\[
Kf(K_n-F)
=
n-1+
\frac{m+1}{n-m-1}
+
\frac{m-1}{n-1}
\]
and
\[
\tau(K_n-F)
=
n^{n-m-2}
(n-m-1)(n-1)^{m-1}.
\]
\end{proposition}

\begin{proof}
Let $v_0$ be the center of $K_{1,m}$, and let
$v_1,\ldots,v_m$ be its leaves.
If the center is listed first, then
\[
L(K_{1,m})
=
\begin{bmatrix}
m&-\boldsymbol{1}_m^T\\
-\boldsymbol{1}_m&I_m
\end{bmatrix}.
\]
It is clear that the all-one vector gives the eigenvalue $0$.

Next, consider the subspace
\[
W
=
\left\{
\begin{bmatrix}
0\\
x
\end{bmatrix}
:
x\in\mathbb{R}^m,\
\boldsymbol{1}_m^Tx=0
\right\}.
\]
For every vector in $W$, we have
\[
L(K_{1,m})
\begin{bmatrix}
0\\
x
\end{bmatrix}
=
\begin{bmatrix}
0\\
x
\end{bmatrix}.
\]
Thus $1$ is an eigenvalue with multiplicity $m-1$.
Moreover, for
\[
z
=
\begin{bmatrix}
m\\
-\boldsymbol{1}_m
\end{bmatrix},
\]
we have
\[
L(K_{1,m})z=(m+1)z.
\]
Hence the Laplacian eigenvalues of $K_{1,m}$ are
\[
0,\
\underbrace{1,\ldots,1}_{m-1\text{ times}},
\ m+1.
\]

Therefore,
\[
\begin{aligned}
Kf(K_n-F)
&=
n-1
+
\frac{m+1}{n-m-1}
+
\sum_{\alpha=1}^{m-1}
\frac{1}{n-1}\\
&=
n-1
+
\frac{m+1}{n-m-1}
+
\frac{m-1}{n-1}.
\end{aligned}
\]
Also,
\[
\begin{aligned}
\tau(K_n-F)
&=
n^{n-2}
\left(
1-\frac{m+1}{n}
\right)
\left(
1-\frac{1}{n}
\right)^{m-1}\\
&=
n^{n-2}
\frac{n-m-1}{n}
\left(
\frac{n-1}{n}
\right)^{m-1}\\
&=
n^{n-m-2}
(n-m-1)(n-1)^{m-1}.
\end{aligned}
\]

The assumption $n\ge m+2$ ensures that $n-m-1>0$.
In particular, after deleting the $m$ edges incident to the center,
at least one edge incident to the center remains, and hence $K_n-F$
is connected.
\end{proof}

%%%%%%%%%%%%%%%%%%%%%%%%%%%%%%%%%%%%%%%%
\begin{proposition}
\label{prop:clique-defect}
Suppose that $H_F\cong K_s$ and
\[
s<n.
\]
Then
\[
Kf(K_n-F)
=
n-1+\frac{s(s-1)}{n-s}
\]
and
\[
\tau(K_n-F)
=
n^{n-s-1}(n-s)^{s-1}.
\]
\end{proposition}

\begin{proof}
The Laplacian matrix of $K_s$ is
\[
L(K_s)=sI_s-J_s.
\]
The all-one vector gives the eigenvalue $0$.
If
\[
x\in\boldsymbol{1}_s^\perp,
\]
then
\[
J_sx=\boldsymbol{0},
\]
and hence
\[
L(K_s)x=sx.
\]
Thus the Laplacian eigenvalues of $K_s$ are
\[
0,\
\underbrace{s,\ldots,s}_{s-1\text{ times}}.
\]

Since
\[
|F|=|E(K_s)|=\binom{s}{2},
\]
Lemma \ref{lem:deletion-spectrum} shows that the nonzero eigenvalues of
the edge-defect matrix are
\[
\underbrace{s,\ldots,s}_{s-1\text{ times}},
\]
and the remaining eigenvalues are zero.
These zero eigenvalues do not contribute to the sum for the Kirchhoff
index and contribute only the factor $1$ to the product for the number
of spanning trees.

Therefore,
\[
\begin{aligned}
Kf(K_n-F)
&=
n-1+
\sum_{\alpha=1}^{s-1}
\frac{s}{n-s}\\
&=
n-1+\frac{s(s-1)}{n-s}.
\end{aligned}
\]
Also,
\[
\begin{aligned}
\tau(K_n-F)
&=
n^{n-2}
\left(
1-\frac{s}{n}
\right)^{s-1}\\
&=
n^{n-2}
\left(
\frac{n-s}{n}
\right)^{s-1}\\
&=
n^{n-s-1}(n-s)^{s-1}.
\end{aligned}
\]

Since $s<n$, there is at least one vertex outside $K_s$.
Even after all edges inside $K_s$ are deleted, every vertex of $K_s$
is adjacent to this outside vertex.
Thus $K_n-F$ is connected.
\end{proof}

%%%%%%%%%%%%%%%%%%%%%%%%%%%%%%%%%%%%%%%%
\subsection{Effective resistance for matching defects}
%%%%%%%%%%%%%%%%%%%%%%%%%%%%%%%%%%%%%%%%

We next give an explicit formula for the effective resistance between
two vertices in the case of a matching deletion.

%%%%%%%%%%%%%%%%%%%%%%%%%%%%%%%%%%%%%%%%
\begin{proposition}
\label{prop:matching-resistance}
Let $G=K_n-F$, and suppose that
\[
H_F\cong mK_2.
\]
Then, for two distinct vertices $x,y\in V(G)$,
\[
R_G(x,y)
=
\frac{2}{n}
+
\frac{\nu_F(x,y)}{n(n-2)},
\]
where
\[
\nu_F(x,y)
=
\begin{cases}
4,
& \text{if } x \text{ and } y \text{ are the two ends of the same deleted edge},\\
2,
& \text{if } x \text{ and } y \text{ are ends of different deleted edges},\\
1,
& \text{if exactly one of } x,y \text{ is an end of a deleted edge},\\
0,
& \text{if neither } x \text{ nor } y \text{ is an end of a deleted edge}.
\end{cases}
\]
\end{proposition}

\begin{proof}
Let
\[
F=\{e_1,\ldots,e_m\}
\]
be the set of deleted edges, and give each edge an arbitrary orientation.
Write
\[
e_\alpha=\{u_\alpha,v_\alpha\},
\]
and put
\[
b_\alpha
=
e_{u_\alpha}-e_{v_\alpha}.
\]
Let
\[
B=
\begin{bmatrix}
b_1&\cdots&b_m
\end{bmatrix},
\qquad
Q=B^TB.
\]

Since
\[
H_F\cong mK_2,
\]
the deleted edges are pairwise disjoint.
Thus
\[
b_\alpha^Tb_\alpha=2
\]
and, for $\alpha\ne\beta$,
\[
b_\alpha^Tb_\beta=0.
\]
Hence
\[
Q=2I_m,
\]
and so
\[
(nI_m-Q)^{-1}
=
\frac{1}{n-2}I_m.
\]

Put
\[
\eta_{xy}
=
B^T(e_x-e_y).
\]
By Theorem \ref{thm:edge-defect-resolvent},
\[
\begin{aligned}
R_G(x,y)
&=
\frac{2}{n}
+
\frac{1}{n}
\eta_{xy}^T
(nI_m-Q)^{-1}
\eta_{xy}\\
&=
\frac{2}{n}
+
\frac{1}{n(n-2)}
\eta_{xy}^T\eta_{xy}\\
&=
\frac{2}{n}
+
\frac{\|\eta_{xy}\|^2}{n(n-2)}.
\end{aligned}
\]
It remains to compute $\|\eta_{xy}\|^2$.

The $\alpha$-th component of $\eta_{xy}$ is
\[
\begin{aligned}
(\eta_{xy})_\alpha
&=
b_\alpha^T(e_x-e_y)\\
&=
(e_{u_\alpha}-e_{v_\alpha})^T(e_x-e_y)\\
&=
\delta_{u_\alpha x}
-\delta_{u_\alpha y}
-\delta_{v_\alpha x}
+\delta_{v_\alpha y}.
\end{aligned}
\]

If $x$ and $y$ are the two ends of the same deleted edge $e_\alpha$,
then
\[
(\eta_{xy})_\alpha=\pm2,
\]
and all other components are zero. Hence
\[
\|\eta_{xy}\|^2=4.
\]

If $x$ and $y$ are ends of two different deleted edges, say
$e_\alpha$ and $e_\beta$, then
\[
(\eta_{xy})_\alpha=\pm1,
\qquad
(\eta_{xy})_\beta=\pm1,
\]
and all other components are zero. Hence
\[
\|\eta_{xy}\|^2=2.
\]

If exactly one of $x,y$ is an end of a deleted edge, then exactly one
component of $\eta_{xy}$ is equal to $\pm1$, and all other components
are zero. Hence
\[
\|\eta_{xy}\|^2=1.
\]

Finally, if neither $x$ nor $y$ is an end of a deleted edge, then
\[
\eta_{xy}=\boldsymbol{0},
\]
and so
\[
\|\eta_{xy}\|^2=0.
\]
Thus
\[
\|\eta_{xy}\|^2=\nu_F(x,y),
\]
and the desired formula follows.
\end{proof}

%%%%%%%%%%%%%%%%%%%%%%%%%%%%%%%%%%%%%%%%
\subsection{Path and cycle defects}
%%%%%%%%%%%%%%%%%%%%%%%%%%%%%%%%%%%%%%%%

Define the polynomial sequence $\{D_k(x)\}_{k\ge0}$ by
\[
D_0(x)=1, 
\qquad
D_1(x)=x, 
\quad
D_k(x)
=xD_{k-1}(x)-D_{k-2}(x)
\qquad
(k\ge2).
\]
This sequence is written in terms of the Chebyshev polynomials of the
second kind as
\[
D_k(x)
=
U_k\left(\frac{x}{2}\right).
\]
Thus the roots of $D_m(x)$ are
\[
2\cos\left(
\frac{j\pi}{m+1}
\right),
\qquad
1\le j\le m.
\]

%%%%%%%%%%%%%%%%%%%%%%%%%%%%%%%%%%%%%%%%
\begin{proposition}
\label{prop:path-defect}
Suppose that $H_F\cong P_{m+1}$.
Then
\[
\tau(K_n-F)
=
n^{n-m-2}D_m(n-2)
\]
and
\[
Kf(K_n-F)
=
n-1-m
+
n\frac{D_m'(n-2)}{D_m(n-2)}.
\]
\end{proposition}

\begin{proof}
The Laplacian eigenvalues of $P_{m+1}$ are
\[
2-2\cos\left(
\frac{j\pi}{m+1}
\right),
\qquad
0\le j\le m.
\]
Hence the positive Laplacian eigenvalues are
\[
\lambda_j
=
2-2\cos\left(
\frac{j\pi}{m+1}
\right),
\qquad
1\le j\le m.
\]

On the other hand, since the roots of $D_m(x)$ are
\[
2\cos\left(
\frac{j\pi}{m+1}
\right),
\qquad
1\le j\le m,
\]
we have
\[
D_m(x)
=
\prod_{j=1}^{m}
\left(
x-
2\cos\left(
\frac{j\pi}{m+1}
\right)
\right).
\]
Therefore,
\[
D_m(n-2)
=
\prod_{j=1}^{m}
\left(
n-2-
2\cos\left(
\frac{j\pi}{m+1}
\right)
\right).
\]
Since the set
\[
\left\{
\cos\left(
\frac{j\pi}{m+1}
\right)
:
1\le j\le m
\right\}
\]
is invariant under changing the sign, we may rearrange the factors and get
\[
\begin{aligned}
D_m(n-2)
&=
\prod_{j=1}^{m}
\left(
n-2+
2\cos\left(
\frac{j\pi}{m+1}
\right)
\right)\\
&=
\prod_{j=1}^{m}(n-\lambda_j).
\end{aligned}
\]

By Theorem \ref{thm:spanning-tree-defect},
\[
\begin{aligned}
\tau(K_n-F)
&=
n^{n-2}
\prod_{j=1}^{m}
\left(
1-\frac{\lambda_j}{n}
\right)\\
&=
n^{n-2}
\prod_{j=1}^{m}
\frac{n-\lambda_j}{n}\\
&=
n^{n-m-2}
\prod_{j=1}^{m}(n-\lambda_j)\\
&=
n^{n-m-2}D_m(n-2).
\end{aligned}
\]

Next, we compute the Kirchhoff index.
Taking the logarithmic derivative of the product expression for $D_m(x)$,
we obtain
\[
\frac{D_m'(x)}{D_m(x)}
=
\sum_{j=1}^{m}
\frac{1}{
x-
2\cos\left(
\dfrac{j\pi}{m+1}
\right)
}.
\]
Using again the sign symmetry of the cosine values, we have, at $x=n-2$,
\[
\frac{D_m'(n-2)}{D_m(n-2)}
=
\sum_{j=1}^{m}
\frac{1}{n-\lambda_j}.
\]
Since
\[
\frac{\lambda_j}{n-\lambda_j}
=
\frac{n}{n-\lambda_j}-1,
\]
we obtain
\[
\begin{aligned}
Kf(K_n-F)
&=
n-1+
\sum_{j=1}^{m}
\frac{\lambda_j}{n-\lambda_j}\\
&=
n-1+
\sum_{j=1}^{m}
\left(
\frac{n}{n-\lambda_j}-1
\right)\\
&=
n-1-m
+
n\sum_{j=1}^{m}
\frac{1}{n-\lambda_j}\\
&=
n-1-m
+
n\frac{D_m'(n-2)}{D_m(n-2)}.
\end{aligned}
\]
\end{proof}

%%%%%%%%%%%%%%%%%%%%%%%%%%%%%%%%%%%%%%%%
\begin{proposition}
\label{prop:cycle-defect}
Let $m\ge3$, and suppose that
\[
H_F\cong C_m.
\]
Then
\[
Kf(K_n-F)
=
n-1+
\sum_{j=1}^{m-1}
\frac{
2-2\cos\left(\dfrac{2\pi j}{m}\right)
}{
n-2+2\cos\left(\dfrac{2\pi j}{m}\right)
}
\]
and
\[
\tau(K_n-F)
=
n^{n-2}
\prod_{j=1}^{m-1}
\left(
1-
\frac{
2-2\cos\left(\dfrac{2\pi j}{m}\right)
}{n}
\right).
\]
\end{proposition}

\begin{proof}
The Laplacian eigenvalues of $C_m$ are
\[
\lambda_j
=
2-2\cos\left(
\frac{2\pi j}{m}
\right),
\qquad
0\le j\le m-1.
\]
Since $\lambda_0=0$, the positive Laplacian eigenvalues are
\[
\lambda_j
=
2-2\cos\left(
\frac{2\pi j}{m}
\right),
\qquad
1\le j\le m-1.
\]

Therefore, by Lemma \ref{lem:deletion-spectrum} and Theorem
\ref{thm:kirchhoff-spanning-defect},
\[
\begin{aligned}
Kf(K_n-F)
&=
n-1+
\sum_{j=1}^{m-1}
\frac{\lambda_j}{n-\lambda_j}\\
&=
n-1+
\sum_{j=1}^{m-1}
\frac{
2-2\cos\left(\dfrac{2\pi j}{m}\right)
}{
n-2+2\cos\left(\dfrac{2\pi j}{m}\right)
}.
\end{aligned}
\]
Also, by Theorem \ref{thm:spanning-tree-defect},
\[
\begin{aligned}
\tau(K_n-F)
&=
n^{n-2}
\prod_{j=1}^{m-1}
\left(
1-\frac{\lambda_j}{n}
\right)\\
&=
n^{n-2}
\prod_{j=1}^{m-1}
\left(
1-
\frac{
2-2\cos\left(\dfrac{2\pi j}{m}\right)
}{n}
\right).
\end{aligned}
\]
\end{proof}

%%%%%%%%%%%%%%%%%%%%%%%%%%%%%%%%%%%%%%%%
\begin{remark}
\label{rem:cycle-connection}
In particular, when $m=n\ge5$, we have
\[
K_n-F=K_n-C_n.
\]
This connects the present result with the study of cycle deletion by
Chair \cite{Chair2012}, distance-class deletion by Tamura
\cite{Tamura2026}, and circulant distance deletion by Tamura and Tanaka
\cite{TamuraTanaka2026}.
A feature of the present method is that the same framework can be applied
also to deletion graphs without assuming cyclic symmetry.
\end{remark}

%%%%%%%%%%%%%%%%%%%%%%%%%%%%%%%%%%%%%%%%
\section{Conclusion and future problems}
\label{sec:conclusion}
%%%%%%%%%%%%%%%%%%%%%%%%%%%%%%%%%%%%%%%%

In this paper, we studied the connected graph
\[
G=K_n-F
\]
obtained from the complete graph $K_n$ by deleting a set $F$ of
$p$ edges.
Letting $B$ be the incidence matrix of the deleted edge set and defining
\[
Q=B^TB
\]
as the edge-defect matrix, we reduced the problems on effective resistance,
the Kirchhoff index, and the number of spanning trees to problems on a
$p\times p$ matrix corresponding to the number of deleted edges.
In particular, we derived a resolvent formula for the effective resistance,
and described the Kirchhoff index and the number of spanning trees
uniformly in terms of the eigenvalues of $Q$.

As a central result of this paper, we obtained the stability identity
\[
\begin{aligned}
Kf(K_n-F)
&=n-1+\frac{2p}{n-2}+
\frac{n}{(n-2)^2}\sum_{\alpha=1}^{p}\frac{(\theta_\alpha-2)^2}{n-\theta_\alpha},
\end{aligned}
\]
where
\[
\theta_1,\ldots,\theta_p
\]
are the eigenvalues of $Q$.
This identity expresses the excess from the Xu--Das--Zhang type lower bound
in terms of the deviations of the eigenvalues of the edge-defect matrix
from $2$.
Moreover, the correction term vanishes if and only if the deleted edge set
$F$ is a matching.

Furthermore, by using majorization, we determined the structure which
minimizes the Kirchhoff index among non-matching deletion graphs under
the assumptions $p\ge2$ and
\[
n\ge\max\{4,2p-1\}.
\]
The minimum is attained only when
\[
H_F\cong P_3\cup(p-2)K_2.
\]
In other words, among all non-matching deleted edge sets, the minimum is
given by the structure in which exactly one pair of deleted edges shares
a vertex.
Also, when $n\ge2p$, the difference between the matching minimum and the
non-matching minimum is
\[
\frac{2n}{(n-1)(n-2)(n-3)}.
\]
This gives the first stability gap.
In particular, this gap is independent of the number $p$ of deleted edges.

Finally, we gave explicit formulas for the Kirchhoff index and the number
of spanning trees for several typical deletion graphs, such as matchings,
stars, cliques, paths, and cycles.
For matching deletions, we also gave an explicit formula for the effective
resistance between each pair of vertices.
These examples show that the edge-defect matrix method works effectively
also for general deleted edge sets without cyclic symmetry.

A natural future problem is to give a further ranking of the graphs $K_n-F$
with small Kirchhoff index for a fixed number $p$ of deleted edges.
In this paper, we determined the minimum attained by a matching deletion
and the minimum among non-matching deleted edge sets.
The next problem is to classify the deletion graphs which give the next
smallest values after
\[
P_3\cup(p-2)K_2,
\]
and to determine the second, third, and further stability gaps together
with their equality conditions.

It would also be natural to use the edge-defect matrix framework to study
extremal problems for the number of spanning trees under a fixed number
of deleted edges, or to rank effective resistances between specified pairs
of vertices.
%Another possible direction is to consider weighted complete graphs,
%or edge deletions from dense graphs other than complete graphs,
%and to study whether similar reduction formulas and stability phenomena
%hold in those settings.

%%%%%%%%%%%%%%%%%%%%%%%%%%%%%%%%%%%%%%%%
\section*{Funding}
%%%%%%%%%%%%%%%%%%%%%%%%%%%%%%%%%%%%%%%%
This research received no external funding.

%%%%%%%%%%%%%%%%%%%%%%%%%%%%%%%%%%%%%%%%
\section*{Data availability}
%%%%%%%%%%%%%%%%%%%%%%%%%%%%%%%%%%%%%%%%
No new data were created or analyzed in this study.
Data sharing is not applicable to this article.

%%%%%%%%%%%%%%%%%%%%%%%%%%%%%%%%%%%%%%%%
\section*{Conflicts of interest}
%%%%%%%%%%%%%%%%%%%%%%%%%%%%%%%%%%%%%%%%
The author declares no conflict of interest.

%%%%%%%%%%%%%%%%%%%%%%%%%%%%%%%%%%%%%%%%%%%%%%%%%%
%%%%%%%%%%%%%%%% References %%%%%%%%%%%%%%%%%%%%%%
%%%%%%%%%%%%%%%%%%%%%%%%%%%%%%%%%%%%%%%%%%%%%%%%%%

\end{document}